\documentclass[a4paper,12pt]{article}

\usepackage{amscd}
\usepackage{amsfonts}
\usepackage{amsmath}
\usepackage{amssymb}
\usepackage{amsthm}
\usepackage[T1]{fontenc} 
\usepackage{here} 
\usepackage{mathrsfs} 
\usepackage{txfonts} 
\usepackage[all]{xy} 
\usepackage{algorithm} 
\usepackage{algpseudocode} 
\usepackage{diagbox} 

\allowdisplaybreaks

\theoremstyle{plain}
\newtheorem{thm}{Theorem}[section]

\newtheorem{prp}[thm]{Proposition}
\newtheorem{crl}[thm]{Corollary}
\newtheorem{conj}[thm]{Conjecture}
\theoremstyle{definition}
\newtheorem{dfn}[thm]{Definition}

\newtheorem{exm}[thm]{Example}

\newcommand{\vs}[1][0.2]{\vspace{#1in}\noindent\ignorespaces}
\newcommand{\ba}{\begin{array*}}
\newcommand{\ea}{\end{array*}}
\newcommand{\be}{\begin{eqnarray*}}
\newcommand{\ee}{\end{eqnarray*}}
\newcommand{\bi}{\begin{itemize}}
\newcommand{\ei}{\end{itemize}}
\newcommand{\bb}{\vs\begin{itembox}}
\newcommand{\eb}{\end{itembox}}
\newcommand{\bc}{\begin{center}}
\newcommand{\ec}{\end{center}}
\newcommand{\bs}{\vs\begin{screen}}
\newcommand{\es}{\end{screen}}

\def\ens#1{{\mathchoice{\left\{ #1 \right\}}{\{ #1 \}}{\{ #1 \}}{\{ #1 \}}}}
\def\set#1#2{{\mathchoice{\left\{ #1 \middle| #2 \right\}}{\{ #1 \mid #2 \}}{\{ #1 \mid #2 \}}{\{ #1 \mid #2 \}}}}
\def\r#1{\text{\rm #1}}

\def\Bigv#1{\left| #1 \right|}
\def\v#1{{\mathchoice{\Bigv{#1}}{| #1 |}{| #1 |}{| #1 |}}}
\def\Bign#1{\left\| #1 \right\|}
\def\n#1{{\mathchoice{\Bign{#1}}{\| #1 \|}{\| #1 \|}{\| #1 \|}}}

\def\ol#1{\overline{#1}{}}

\newcommand{\bC}{\mathbb{C}}

\newcommand{\bN}{\mathbb{N}}

\newcommand{\bP}{\mathbb{P}}
\newcommand{\bQ}{\mathbb{Q}}
\newcommand{\bR}{\mathbb{R}}

\newcommand{\bZ}{\mathbb{Z}}

\newcommand{\cA}{\mathscr{A}}

\newcommand{\cC}{\mathscr{C}}

\newcommand{\cP}{\mathscr{P}}

\newcommand{\cU}{\mathscr{U}}

\newcommand{\C}{\bC}

\newcommand{\N}{\bN}
\newcommand{\Q}{\bQ}
\newcommand{\R}{\bR}
\newcommand{\Z}{\bZ}

\newcommand{\Fp}{\mathbb{F}_p}

\newcommand{\hatZ}{\widehat{\mathbb{Z}}{}}

\newcommand{\ch}{\r{ch}}

\newcommand{\id}{\r{id}}
\newcommand{\im}{\r{im}}

\newcommand{\Spec}{\r{Spec}}

\algnewcommand\algorithmicbreak{{\bf break}}
\algnewcommand\Break{\algorithmicbreak{}}
\algnewcommand\algorithmiccontinue{{\bf continue}}
\algnewcommand\Continue{\algorithmiccontinue{}}

\title{Notes on Algebraic Properties and Non-Standard Analysis of the Ring of Integers Modulo Infinitely Large Primes}
\author{Tomoki Mihara}
\date{}

\begin{document}

\maketitle
\begin{abstract}
As applications of non-standard analysis to the study of transcendence in the ring $\mathscr{A}$ of integers modulo infinitely large primes, we extend transcendence criteria by Anzawa--Funakura and Matsusaka--Seki. Using the extension of the criterion by Anzawa--Funakura, we observe conditional relation between asymptotic behaviour of prime gaps and transcendence in $\mathscr{A}$. We also summarise known algebraic and model theoretic results on $\mathscr{A}$ for number theorists, and share topics in transcendental number theory with algebraists and model theorists.
\end{abstract}

\tableofcontents

\section{Introduction}
\label{Introduction}

Let $\bP$ denote the set of prime numbers, and set
\be
\cA \coloneqq \prod_{p \in \bP} \Fp \bigg/ \bigoplus_{p \in \bP} \Fp.
\ee
In number theory, the ring $\cA$ plays a role of the base ring in the study of finite multiple zeta values, which are finite analogues of multiple zeta values introduced by D.\ Zagier. Starting from \cite{KZ}, there have been various studies on specific elements of $\cA$ by number theorists, but not so many about the algebraic structure of $\cA$ other than the well-known fact that $\cA$ is a $\Q$-algebra are mentioned in those works.

\vs
On the other hand, groups and rings over the form $\prod/\bigoplus$ and the residue fields of such rings at their maximal ideals have ever been deeply studied in algebra and model theory from various other motivations, e.g.\ studies of slender groups, Fuchs-44-groups, ultrapowers of Banach spaces and Banach algebras, $P$-points in $\beta \N \setminus \N$, and so on. Although algebraists and model theorists have ever revealed several properties of $\cA$, it seems that they do not connect those results to number theoretic objects.

\vs
The aim of this paper is to bridge those distinct branches of mathematics in order to accelerate the study of $\cA$. In particular, we summarise known algebraic and model theoretic results on the ring $\cA$ for number theorists, and topics in transcendental number theory with algebraists and model theorists. In order to show application of non-standard analysis to the study of transcendence, we extend transcendence criteria \cite{MS26} Lemma 2.1 (cf.\ \cite{AF24} Proposition 3.7) and \cite{MS26} Lemma 2.3. Although the extensions are not so significant compared to the novelty of the original works, we expect that they help the reader to recognise how to extend other arguments in similar ways. Using the extension of the criterion by Anzawa--Funakura, we prove that the image of the sequence of the $k$-th least prime number above $p \in \bP$ is transcendental in $\cA$ for any $k \in \N_{> 0}$ under a number theoretic conjecture on asymptotic behaviour of prime gaps which follows from Cram\'er's conjecture (cf.\ \cite{Cra36} (4)).

\vs
We briefly explain contents of this paper. In \S \ref{Convention}, we introduce convention for this paper. Especially, we recall an ultrafilter and an ultraproduct for the reader who is not familiar with non-standard analysis. In \S \ref{Global Structure}, we summarise known algebraic and model theoretic results on the ring $\cA$. In \S \ref{Local Structure}, we study $\C$-rational points of $\cA$, and introduce several notions on transcendence of elements of $\cA$. In \S \ref{Transcendence Criteria}, we give extensions of transcendence criteria \cite{MS26} Lemma 2.1 and \cite{MS26} Lemma 2.3. In \S \ref{Relation to Prime Gaps}, we introduce convention for prime gaps, and prove the transcendence of the image of the sequence of the $k$-th least prime number above $p \in \bP$ in $\cA$ for any $k \in \N_{> 0}$ under the number theoretic conjecture.

\section{Convention}
\label{Convention}

We denote by $\aleph_0$ the least infinite cardinal, which is identical to the set $\N$ of non-negative integers, and by $\aleph$ the cardinality $2^{\aleph_0}$ of the continuum.

\vs
For sets $X$ and $Y$, we denote by $X^Y$ the set of maps $Y \to X$. When we handle a sequence or a family $s$ indexed by a set $I$, we frequently use the map notation $s(i)$ instead of the subscript notation $s_i$ to point the entry at $i \in I$, in order to avoid massive use of subscripts.

\vs
For a set $X$, $x \in X$, and a binary relation $R$ on $X$, we set $X_{R x} \coloneqq \set{x' \in X}{x' R x}$. We note that every $d \in \N$ is identical to $\N_{< d}$, and hence for a set $X$, $X^d$ formally means $X^{\N_{< d}}$, which is naturally identified with the set of $d$-tuples in $X$.

\vs
For a set $I$, we denote by $\# I$ its cardinality, and by $\cP(I)$ the set of subsets of $I$. For a set $I$ and a binary relation $R$ on cardinal numbers, we set $\cP_{R \aleph_0}(I) \coloneqq \set{U \in \cP(I)}{\# U R \aleph_0}$.

\vs
Let $I$ be a set. A $\cU \in \cP(\cP(I))$ is said to be an {\it ultrafilter} on $I$ if it satisfies the following:
\bi
\item[(1)] For any $U \in \cU$ and any $U' \in \cP(I)$, $U \subset U'$ implies $U' \in \cU$.
\item[(2)] For any $(U,U') \in \cU^2$, $U \cap U' \in \cU$.
\item[(3)] For any $U \in \cU$, precisely one of $U \in \cU$ and $I \setminus U \in \cU$ holds.
\ei
We denote by $\beta I$ the set of ultrafilters on $I$, and equip it with the Stone topology, i.e.\ the topology generated by $\set{\set{U' \in \cU}{U \subset U'}}{U \in \cP(I)}$. For any $i \in I$, $\set{U \in \cP(I)}{i \in U}$ is an ultrafilter on $I$, and such an ultrafilter is called a {\it principal ultrafilter}. we identify $I$ with the subset of $\beta I$ consisting of principal ultrafilters. Then $\beta I$ is a totally disconnected compact Hausdorff topological space, and $I$ is an open subspace of $\beta I$. In particular, $\beta I \setminus I$ is a totally disconnected compact Hausdorff topological space.

\vs
For $\cU \in \beta I$ and a family $(S_i)_{i \in I}$ of sets indexed by $I$, we denote by $\lim_{i \to \cU} S_i$ the quotient of $\prod_{i \in I} S_i$ by the equivalence relation $=_{\cU}$ defined by
\be
f =_{\cU} f' \stackrel{\r{def}}{\Leftrightarrow} \set{i \in I}{f(i) = f'(i)} \in \cU.
\ee
If $(S_i)_{i \in I}$ is a family of Abelian groups (resp.\ rings, fields, ordered sets), then $\lim_{i \to \cU} S_i$ naturally forms an Abelian group (resp.\ a ring, a field, an ordered set) by {\L}o\'s's theorem (cf.\ \cite{BS69} Theorem 5/2.1).

\vs
For $f \in \prod_{i \in I} S_i$, we denote by $\lim_{i \to \cU} f(i)$ the image of $f$ in $\lim_{i \to \cU} S_i$. If $(S_i)_{i \in I}$ is a family of Abelian groups and $\cU$ is non-principal, then the canonical projection
\be
\prod_{i \in I} S_i & \twoheadrightarrow & \lim_{i \to \cU} S_i \\
f & \mapsto & \lim_{i \to \cU} f(i)
\ee
is a group homomorphism whose kernel contains $\bigoplus_{i \in I} S_i$. Therefore, it induces a surjective group homomorphism $(\prod_{i \in I} S_i) / (\bigoplus_{i \in I} S_i) \twoheadrightarrow \lim_{i \to \cU} S_i$. For $f \in (\prod_{i \in I} S_i) / (\bigoplus_{i \in I} S_i)$, we abuse the notation $\lim_{i \to \cU} f(i)$ to denote the image of $f$ in $\lim_{i \to \cU} S_i$.

\vs
Let $S$ be a set. We identify $S$ with the image of the composite $S \hookrightarrow \lim_{p \to \cU} S$ of the diagonal embedding $S \hookrightarrow S^{\bP}$ and the canonical projection $S^{\bP} \twoheadrightarrow \lim_{p \to \cU} S$.

\section{Global Structure}
\label{Global Structure}

In this section, we recall known properties of the algebraic structure of $\cA$.

\begin{dfn}
We say that a ring is a {\it field boundary} if it is isomorphic to $(\prod_{i \in I} k_i) / (\bigoplus_{i \in I} k_i)$ for a family $(k_i)_{i \in I}$ of fields indexed by a set $I$.
\end{dfn}

Since $\cA$ is a field boundary, $\cA$ satisfy all algebraic properties of field boundaries. Therefore, we recall basic properties of field boundaries.

\vs
A ring $R$ is said to be {\it von Neumann regular} if for any $r \in R$, there exists some $s \in R$ such that $r = rsr$. Although the notion of von Neumann regularity works also in non-commutative ring theory, we always consider a ring to be commutative.

\begin{prp}
\label{von Neumann regular}
Every field boundary is von Neumann regular. In particular, $\cA$ is von Neumann regular.
\end{prp}

\begin{proof}
Every field is von Neumann regular, the direct product of von Neumann regular rings is again von Neumann regular, and the quotient of a von Neumann regular ring by an ideal is again von Neumann regular.
\end{proof}

As benefit to consider von Neumann regularity, we obtain many ring theoretic properties of $\cA$:

\begin{crl}
\label{absolutely flat}
Let $A$ be a field boundary, e.g.\ $\cA$.
\bi
\item[(1)] Every $A$-module is flat.
\item[(2)] Every finitely generated ideal of $A$ is generated by an idempotent, and hence $A$ is a B\'ezout ring.
\item[(3)] Every prime ideal of $A$ is a maximal ideal, and the localisation of $A$ at any prime ideal is a field.
\item[(4)] The ring $A$ is reduced, and hence is isomorphic to a subring of the direct product of fields.
\item[(5)] The spectrum $\Spec(A)$ of $A$ is a totally disconnected compact Hausdorff topological space with respect to the Zariski topology, and hence the constructible topology coincides with the Zariski topology.
\ei
\end{crl}

\begin{proof}
The assertions follow from Proposition \ref{von Neumann regular} (cf.\ \cite{AM69} Exercise 2/27, 2/28, 3/10, 3/11, 3/30, and 4/3).
\end{proof}

We recall the well-known explicit description of prime ideals of a field boundary.

\begin{prp}
\label{characterisation of Spec}
Let $(k_i)_{i \in I}$ be a family of fields indexed by a set $I$. Set $A \coloneqq (\prod_{i \in I} k_i) / (\bigoplus_{i \in I} k_i)$.
\bi
\item[(1)] For any $\cU \in \beta I \setminus I$, $m_{\cU} \coloneqq \set{f \in A}{\lim_{i \to \cU} f(i) = 0]}$ is a maximal ideal of $A$.
\item[(2)] The map
\be
\beta I \setminus I & \to & \Spec(A) \\
\cU & \mapsto & m_{\cU}
\ee
is a homeomorphism with respect to the Zariski topology.
\ei
\end{prp}

\begin{proof}
(1) Since $m_{\cU}$ is the kernel of the canonical projection $A \twoheadrightarrow \lim_{i \to \cU} k_i$, it is a maximal ideal.

\vs
(2) For a $U \in \cP(I)$, we denote by $e_U \in A$ the image of the characteristic function of $U$. The first assertion of Corollary \ref{absolutely flat} (2) implies that the correspondence $m \mapsto \set{U \in \cP(I)}{1 - e_U \in m}$ defines the continuous inverse of the given map.
\end{proof}

Using the interpretation, we will study residue fields of $\cA$ in \S \ref{Local Structure}. The interpretation will be used also in our future studies on $\cA$.

\begin{dfn}
We say that an Abelian group is an {\it Abelian group boundary} (resp.\ a {\it cyclic group boundary}) if it is isomorphic to $(\prod_{i \in I} C_i) / (\bigoplus_{i \in I} C_i)$ for a family $(C_i)_{i \in I}$ of Abelian (resp.\ cyclic) groups indexed by a countable infinite set $I$.
\end{dfn}

Since the additive group $\cA = (\prod_{p \in \bP} \Z/p \Z) / (\bigoplus_{p \in \bP} \Z/p \Z)$ and the multiplicative group $\cA^{\times} \cong (\prod_{p \in \bP} \Z/(p-1)\Z) / (\bigoplus_{p \in \bP} \Z/(p-1)\Z)$ are cyclic group boundaries, $\cA$ and $\cA^{\times}$ satisfy all algebraic properties of cyclic group boundaries. Therefore, we recall basic properties of cyclic group boundaries.

\vs
An Abelian group $A$ is said to be {\it reduced} if it admits no non-zero divisible element, and is said to be a {\it Fuchs-44-group} if for any family $(R_j)_{j \in J}$ of reduced Abelian groups indexed by a set $J$ and any group homomorphism $f \colon A \to \bigoplus_{j \in J} R_j$, there exists some $(m,J') \in \Z_{> 0} \times \cP_{< \aleph_0}(J)$ such that $f(ma)(j) = 0$ for any $(a,j) \in A \times (J \setminus J')$.

\vs
The notion of Fuchs-44-group is named after \cite{Fuc70} Chapter IX Problem 44 asking to investigate an Abelian group $A$ such that for any embedding of $A$ into the direct sum of reduced Abelian groups, there exists some $n \in \N_{> 0}$ such that the image of $nA$ is contained in a direct summand given as a finite direct sum, and is deeply studied in \cite{Iva78}.

\begin{prp}
Every cyclic group boundary is a Fuchs-44-group. In particular, $\cA$ and $\cA^{\times}$ are Fuchs-44-groups.
\end{prp}

We note that the assertion for $\cA$ is obvious by the divisibility of $\cA$. See also \cite{Eda83} Corollary 4 and \cite{Fuc15} Chapter 6 \S 1 Notes for more general criteria of Fuchs-44-groups.

\begin{proof}
Since the image of a Fuchs-44-group by a group homomorphism is again a Fuchs-44-group, the assertion immediately follows from the fact that $\Z^{\N}$ is a Fuchs-44-group (cf.\ \cite{Eda83} Lemma 1).
\end{proof}

For $n \in \N$, we denote by $e_n \in \Z^{\N}$ the characteristic function of $\ens{n}$. An Abelian group $A$ is said to be a {\it slender group} if for any group homomorphism $f \colon \Z^{\N} \to A$, there exists some $n_0 \in \N$ such that for any $n \in \N_{\geq n_0}$, $f(e_n) = 0$ holds, and is said to be an {\it almost-slender group} if for any group homomorphism $f \colon \Z^{\N} \to A$, there exists some $(m,n_0) \in \Z_{> 0} \times \N$ such that for any $n \in \N_{\geq n_0}$, $f(me_n) = 0$ holds (cf.\ \cite{GRW81} Theorem 3.1 and \cite{Eda83} Definition 1).

\vs
We note that an Abelian group of cardinality $< \aleph$ is slender if and only if it is torsion-free and reduced (cf.\ \cite{Fuc15} \S 13/2 (e) and Lemma 13/2.3), and every slender group is an almost-slender group. The notion of slender groups is originated from studies by J.\ {\L}o\'s in 1950s, and is connected to various results related to Specker's theorem (cf.\ \cite{Spe50}), e.g.\ {\L}o\'s's theorem (cf.\ \cite{EL54}, \cite{Zee55}) and Eda's theorem (cf.\ \cite{Eda82} Corollary 2).

\begin{prp}
\bi
\item[(1)] For any divisible cyclic group boundary $A$, there is no non-zero group homomorphism from $A$ to an almost-slender group. In particular, there is no non-zero group homomorphism from $\cA$ to an almost-slender group.
\item[(2)] For any cyclic group boundary $A$, there is no non-zero group homomorphism from $A$ to a slender group. In particular, there is no non-zero group homomorphism from $\cA^{\times}$ to a slender group.
\ei
\end{prp}

\begin{proof}
Since $A$ is a cyclic group boundary, it is reduced to the case where $A = (\prod_{i \in I} C_i) / (\bigoplus_{i \in I} C_i)$ for a family $(C_i)_{i \in I}$ of cyclic groups indexed by a countable infinite set $I$. We denote by $\pi$ the canonical projection $\prod_{i \in I} C_i \twoheadrightarrow A$. 

\vs
(1) It suffices to show that for any almost-slender group $S$, every group homomorphism $f \colon A \to S$ is zero. Since $S$ is almost-slender and $I$ is countable, there exists some $(m,I') \in \Z_{> 0} \times \cP_{< \aleph_0}(I)$ such that for any $c \in \prod_{i \in I} C_i$ such that $c |_{I'} = 0$, $(f \circ \pi)(mc) = 0$ holds (cf.\ \cite{Eda83} Theorem 1). By $\ker(\pi) = \bigoplus_{i \in I} C_i$, we obtain $(f \circ \pi)(mc) = 0$ for any $c \in \prod_{i \in I} C_i$. By the surjectivity of $\pi$, we obtain $mf = 0$. Since $A$ is divisible, we obtain $f = 0$.

\vs
(2) It suffices to show that for any slender group $S$, every group homomorphism $f \colon A \to S$ is zero. Since $S$ is slender and $I$ is countable, there exists some $I' \in \cP_{< \aleph_0}(I)$ such that for any $c \in \prod_{i \in I} C_i$ such that $c |_{I'} = 0$, $(f \circ \pi)(c) = 0$ holds (cf.\ \cite{Eda82} Corollary 2). By $\ker(\pi) = \bigoplus_{i \in I} C_i$, we obtain $(f \circ \pi)(c) = 0$ for any $c \in \prod_{i \in I} C_i$. By the surjectivity of $\pi$, we obtain $f = 0$.
\end{proof}

\begin{crl}
For any cyclic group boundary $A$, there is no non-zero group homomorphism $A \to \Z$. In particular, there is no non-zero group homomorphism $\cA^{\times} \to \Z$.
\end{crl}

\begin{proof}
The assertion immediately follows from the slenderness of $\Z$ (cf.\ \cite{Fuc15} Example 13/2.5).
\end{proof}

An Abelian group $A$ is said to be an {\it algebraically compact group} if there exists some Abelian group $B$ such that $A \times B$ admits a compact Hausdorff topology compatible with the group structure. In particular, the underlying group of a compact Hausdorff topological Abelian group is an algebraically compact group.

\vs
We note that an Abelian group $A$ is algebraically compact if and only if for any injective group homomorphism $f \colon A \hookrightarrow B$ into an Abelian group $B$, if $f(A)$ is a pure subgroup of $B$, i.e.\ $nf(A) = f(A) \cap nB$ holds for any $n \in \N$, then $f(A)$ is a direct summand of $B$ (cf.\ \cite{Fuc15} Theorem 1.2/6). An Abelian group satisfying this property is also called a {\it pure-injective group}, and is deeply studied in a way parallel to an injective $\Z$-module (cf.\ \cite{Mar60} Introduction and \cite{Fuc15} Chapter 6 Abstract).

\begin{prp}
Every Abelian group boundary is an algebraically compact group. In particular, $\cA$ and $\cA^{\times}$ are algebraically compact groups.
\end{prp}

We note that the quotient topology on $\cA$ and the relative topology on $\cA^{\times}$ are not Hausdorff, and hence the assertion for $\cA$ and $\cA^{\times}$ is not obvious from the definitions.

\begin{proof}
The assertion is precisely the same as \cite{Fuc15} Corollary 6/12.
\end{proof}

\section{Local Structure}
\label{Local Structure}

In this section, we recall known properties of the algebraic structure of residue fields of $\cA$, and formulate several new transcendence in $\cA$ using non-standard analysis.

\vs
A field is said to be {\it quasi-finite} if it is perfect and its absolute Galois group is isomorphic to $\hatZ$. Every finite field is quasi-finite, but a quasi-finite field is not necessarily of characteristic $> 0$. Indeed, we have the following:

\begin{prp}
\label{quasi-finite}
For any $m \in \Spec(\cA)$, $\cA/m$ is a quasi-finite field such that the multiplicative group $(\cA/m)^{\times}$ naturally forms a module over non-standard integers with $(\cA/m)^{\times}/((\cA/m)^{\times})^2 \cong \Z/2 \Z$, and is isomorphic to a subfield of $\C$.
\end{prp}

\begin{proof}
The first assertion immediately follows from Proposition \ref{characterisation of Spec}, \cite{Ngu24} Corollary 3.28, and {\L}o\'s's theorem (cf.\ \cite{BS69} Theorem 5/2.1). The second assertion immediately follows from $\# \cA/m \leq \# \cA = \aleph$ and $\ch(\cA/m) = 0$.
\end{proof}

By the structure of the multiplicative group, $(x^2-f)(x^2-g)(x^2-fg) \in (\cA/m)[x]$ has a root in $\cA/m$ for any $(f,g) \in (\cA/m)^2$. By Euler's criterion, the action of non-standard integers on the multiplicative group satisfies
\be
f^{\frac{x-1}{2}} =
\left\{
\begin{array}{ll}
0 & (f = 0) \\
1 & (f \in ((\cA/m)^{\times})^2) \\
-1 & (f \in (\cA/m)^{\times} \setminus ((\cA/m)^{\times})^2)
\end{array}
\right.
\ee
for any $f \in \cA/m$, where $x$ denotes the non-standard odd integer $\lim_{p \to \cU} p \in \lim_{p \to \cU} \Z$ and $\cU$ denotes the non-principal ultrafilter on $\bP$ corresponding to $m$ by Proposition \ref{characterisation of Spec} (2). In this way, first order properties generically true for $\Fp$ with $p \in \bP$ are reflected by the ultra-product $\lim_{p \to \cU}$.

\begin{crl}
Gel'fand transform
\be
\cA & \to & \C^{\Spec(\cA)(\C)} \\
f & \mapsto & (\phi(f))_{\phi \in \Spec(\cA)(\C)}
\ee
is an injective non-surjective $\Q$-algebra homomorphism, where $\Spec(\cA)(\C)$ denotes the set of $\C$-rational points of $\Spec(\cA)$.
\end{crl}

\begin{proof}
The assertion immediately follows from Corollary \ref{absolutely flat} (4) and the second assertion of Proposition \ref{quasi-finite}.
\end{proof}

\begin{crl}
\label{enough points}
The set $\Spec(\cA)(\C)$ of $\C$-rational points of $\Spec(\cA)$ is of cardinality $2^{\aleph}$.
\end{crl}

\begin{proof}
By $\# \cA = \aleph$, we have $\# \Spec(\cA)(\C) \leq 2^{\aleph}$. By Proposition \ref{characterisation of Spec} (2) and the second assertion of Proposition \ref{quasi-finite}, we have $\# \Spec(\cA)(\C) \geq \# \Spec(\cA) = \#(\beta \bP \setminus \bP) = 2^{\aleph}$.
\end{proof}

\begin{dfn}
An $f \in \cA$ is said to be {\it transcendental at $\phi \in \Spec(\cA)(\C)$} if $\phi(f) \in \C$ is a transcendental number, and is said to be {\it everywhere transcendental} if it is transcendental at any $\phi \in \Spec(\cA)(\C)$.
\end{dfn}

\begin{prp}
\label{everywhere transcendental characterisation}
For any $f \in \cA$, $f$ is everywhere transcendental if and only if $F(f) \in \cA^{\times}$ for any $F \in \Q[x] \setminus \ens{0}$.
\end{prp}

\begin{proof}
Suppose that $f$ is everywhere transcendental. Let $F \in \Q[x] \setminus \ens{0}$. For any maximal ideal $m \subset \cA$, there exists some $\phi \in \Spec(\cA)(\C)$ such that $\ker(\phi) = m$ by the second assertion of Proposition \ref{quasi-finite}, and hence $F(f) \notin m$ by the transcendence of $\phi(f)$. This implies $F(f) \in \cA^{\times}$.

\vs
Suppose that $F(f) \in \cA^{\times}$ for any $F \in \Q[x] \setminus \ens{0}$. Let $\phi \in \Spec(\cA)(\C)$. For any $F \in \Q[x] \setminus \ens{0}$, we have $F(\phi(f)) = \phi(F(f)) \in \C^{\times}$ by $F(f) \in \cA^{\times}$. Therefore $\phi(f)$ is a transcendental number.
\end{proof}

An $f \in \cA$ is said to be {\it naively transcendental} if $F(f) \neq 0$ holds for any $F \in \Q[x] \setminus \ens{0}$. By Proposition \ref{everywhere transcendental characterisation}, every everywhere transcendental element of $\cA$ is naively transcendental. More strongly, ``everywhere'' can be replaced by ``somewhere'' in the following sense:

\begin{prp}
\label{transcendence implies naive transcendence}
Let $f \in \prod_{p \in \bP} \Fp$. Then the following are equivalent:
\bi
\item[(1)] The image of $f$ in $\cA$ is naively transcendental.
\item[(2)] The image of $f$ in $\cA$ is transcendental at some $\phi \in \Spec(\cA)(\C)$.
\item[(3)] The image of $f$ in $\cA$ is transcendental at $2^{\aleph}$-many $\phi \in \Spec(\cA)(\C)$.
\item[(4)] There exists some $U \in \cP_{= \aleph_0}(\bP)$ such that for any $F \in \Q[x] \setminus \ens{0}$, there exists some $U_0 \in \cP_{< \aleph_0}(U)$ such that for any $p \in U \setminus U_0$, $F \in \Z_{(p)}[x]$ and $F(f(p)) \neq 0$ hold.
\ei
\end{prp}

\begin{proof}
The implication from (3) to (2) follows from $1 < 2^{\aleph}$. The implication from (2) to (1) follows from the commutativity of $\phi$ and a polynomial function over $\Q$. The implication from (1) to (4) follows from the countability of the multiplicative set $\Q[x] \setminus \ens{0}$ and the existence of a pseudo-intersection.

\vs
Assume (4). We denote by $\ol{f} \in \cA$ the image of $f$, and by $X \subset \Spec(\cA)$ the clopen subset corresponding to $\set{\cU \in \beta \bP \setminus \bP}{U \in \cU} \subset \beta \bP \setminus \bP$ by the homeomorphism in Proposition \ref{characterisation of Spec} (2). Set $Y \coloneqq \set{\phi \in \Spec(\cA)(\cC)}{\ker(\phi) \in X}$. By Corollary \ref{enough points} and the second assertion of Proposition \ref{quasi-finite}, we have
\be
2^{\aleph} = \# \Spec(\cA)(\C) \geq \# Y \geq \# X = \#(\beta U \setminus U) = 2^{\aleph},
\ee
and hence $\# Y = 2^{\aleph}$. It suffices to show that $\ol{f}$ is transcendental at any $\phi \in Y$.

\vs
Let $F \in \Q[x] \setminus \ens{0}$. It suffices to show $F(\ol{f}) \notin \ker(\phi)$. By the choice of $U$, there exists some $U_0 \in \cP_{< \aleph_0}(U)$ such that for any $p \in U \setminus U_0$, $F \in \Z_{(p)}[x]$ and $F(f(p)) \neq 0$ hold. By $\# U = \aleph_0 > \# U_0$, we have $\phi(F(\ol{f})) \in \C^{\times}$, i.e.\ $F(\ol{f}) \notin \ker(\phi)$.
\end{proof}

We further introduce a hierarchy of stronger transcendence associated to the growth rates of sequences in terms of ultra-polynomial (cf.\ \cite{Ngu24} Definition 3.2). For this purpose, we prepare a notation from non-standard analysis.

\vs
Let $\cU \in \beta \bP \setminus \bP$. We define a preorder $\leq_{\cU}$ on $\R^{\bP}$ by
\be
s \leq_{\cU} t \stackrel{\r{def}}{\Leftrightarrow} \set{p \in \bP}{s(p) \leq t(p)} \in \cU,
\ee
i.e.\ the pull-back of the order of the set $\lim_{p \to \cU} \R$ of non-standard real numbers. For $(s,t) \in \R^X \times \R^Y$ with $(X,Y) \in \cP(\R)^2$, we define
\be
s \in O(t) & \stackrel{\r{def}}{\Leftrightarrow} & \exists c \in \R_{> 0}[\exists x_0 \in X \cap Y[\forall x \in (X \cap Y)_{\geq x_0}[\v{s(x)} \leq ct(x)]]] \\
s \in o(t) & \stackrel{\r{def}}{\Leftrightarrow} & \forall c \in \R_{> 0}[\exists x_0 \in X \cap Y[\forall x \in (X \cap Y)_{\geq x_0}[\v{s(x)} \leq ct(x)]]].
\ee
Similarly, for $(s,t) \in (\R^{\bP})^2$, we define
\be
s \in_{\cU} O(t) & \stackrel{\r{def}}{\Leftrightarrow} & \exists c \in \R_{> 0}[\v{s} \leq_{\cU} ct] \\
s \in_{\cU} o(t) & \stackrel{\r{def}}{\Leftrightarrow} & \forall c \in \R_{> 0}[\v{s} \leq_{\cU} ct].
\ee
The following connects the asymptotic relations $\in O$ and $\in o$ with the non-standard relations $\in_{\cU} O$ and $\in_{\cU} o$:

\begin{prp}
\label{domination}
For any $(s,t) \in (\R^{\bP})^2$, $s \in O(t)$ implies $s \in_{\cU} O(t)$, and $s \in o(t)$ implies $s \in_{\cU} o(t)$.
\end{prp}

\begin{proof}
The assertions immediately follow from $\bP_{\geq p_0} \in \cU$ for any $p_0 \in \bP$.
\end{proof}

For $c \in \Fp$ with $p \in \bP$, we denote by $R_p(c)$ the unique representative of $c$ in $\Z \cap (-p/2,p/2]$. For $F \in R[x]$ with a ring $R$ and $d \in \N$, we denote by $a_d(F) \in R$ the coefficient of $F$ at degree $d$. For $F \in \prod_{p \in \bP} (\Fp[x])$, we denote by $\n{F} \in \N^{\bP}$ and $\deg(F) \in \N^{\bP}$ the maps defined by
\be
\n{F}(p) & \coloneqq & \max \set{\v{R_p(a_d(F(p)))}}{d \in \N} \\
\deg(F)(p) & \coloneqq & \deg(F(p)),
\ee
where $\deg(F(p))$ denotes the degree of $F(p) \in \Fp[x]$, which is defined as $0$ when $F(p) = 0$. For $(B,D) \in (\R_{\geq 0}^{\bP})^2$, we set
\be
P_{B,D,\cU} \coloneqq \set{F \in \prod_{p \in \bP} (\Fp[x])}{\n{F} \in_{\cU} O(B) \land \deg(F) \in_{\cU} O(D)}.
\ee
For $F \in \Z[x]$, we denote by $F_{\cA} \in \prod_{p \in \bP} (\Fp[x])$ the map
\be
F_{\cA}(p) \coloneqq \sum_{d=0}^{\infty} (a_d(F) + p \Z) x^d.
\ee
We note that if $B$ and $D$ are non-zero constant maps, then $P_{B,D,\cU}$ is essentially the set of polynomials in $\Z$ in the following sense:

\begin{prp}
\label{BD-polynomial}
For any $(B,D) \in (\R_{\geq 0}^{\bP})^2$, if $\lim_{p \to \cU} B(p) \in \R_{> 0}$ and $\lim_{p \to \cU} D(p) \in \R_{> 0}$, then $P_{B,D,\cU} = \set{F \in \prod_{p \in \bP} (\Fp[x])}{\exists G \in \Z[x][F =_{\cU} G_{\cA}]}$.
\end{prp}

\begin{proof}
We set $P' \coloneqq \set{F \in \prod_{p \in \bP} (\Fp[x])}{\exists G \in \Z[x][F =_{\cU} G_{\cA}]}$, $B_{\infty} \coloneqq \lim_{p \to \cU} B(p)$, and $D_{\infty} \coloneqq \lim_{p \to \cU} D(p)$.

\vs
Let $F \in P_{B,D,\cU}$. By $B_{\infty} \in \R_{> 0}$ and $\n{F} \in_{\cU} O(B)$, we have
\be
\set{p \in \bP}{\forall d \in \N[\v{R_p(a_d(F(p)))} \leq c_0 B_{\infty}]} \in \cU
\ee
for some $c_0 \in \R_{> 0}$. Since $\Z \cap [-c_0 B_{\infty}, c_0 B_{\infty}]$ is a finite set, we have $\lim_{p \to \cU} R_p(a_d(F(p))) \in \Z \cap [-c_0 B_{\infty}, c_0 B_{\infty}]$. By $D_{\infty} \in \R_{> 0}$ and $\deg(F) \in_{\cU} O(B)$, we have
\be
\set{p \in \bP}{\deg(F(p)) \leq c_1 D_{\infty}]} \in \cU
\ee
for some $c_1 \in \R_{> 0}$. Since $\N \cap \R_{\leq c_1 D_{\infty}}$ is a finite set, we have $\lim_{p \to \cU} \deg(F(p)) \in \N \cap \R_{\leq c_1 D_{\infty}}$. We obtain $F \in P'$.

\vs
Let $F \in P'$. Take $G \in \Z[x]$ such that $F =_{\cU} G_{\cA}$. Set $B' \coloneqq \max \set{\v{a_d(G)}}{d \in \N}$. Then we have $\n{F} \leq_{\cU} B' =_{\cU} (B_{\infty}^{-1}B')B$ and $\deg(F) \leq_{\cU} \deg(G) =_{\cU} (D_{\infty}^{-1}\deg(G))D$. This implies $F \in P_{B,D,\cU}$.
\end{proof}

By Proposition \ref{BD-polynomial}, the following gives a notion generalising the everywhere transcendence:

\begin{dfn}
Let $(B,D) \in (\R_{\geq 0}^{\bP})^2$. An $f \in \cA$ is said to be {\it $(B,D)$-transcendental at $\cU \in \beta \bP \setminus \bP$} if $(\lim_{p \to \cU} F(p))(\lim_{p \to \cU} f(p)) \neq 0$ for any $F \in P_{B,D,\cU}$ with $\lim_{p \to \cU} F(p) \neq 0$, and is said to be {\it everywhere $(B,D)$-transcendental} if it is $(B,D)$-transcendental at any $\cU \in \beta \bP \setminus \bP$.
\end{dfn}

We show that the parametrised transcendence forms a hierarchy along the growth rate.

\begin{prp}
\label{hierarchy}
Let $(B,D) \in (\R_{\geq 0}^{\bP})^2$ and $(B',D') \in (\R_{\geq 0}^{\bP})^2$.
\bi
\item[(1)] For any $\cU \in \beta \bP \setminus \bP$, if $B' \in_{\cU} O(B)$ and $D' \in_{\cU} O(D)$, then every element of $\cA$ $(B,D)$-transcendental at $\cU$ is $(B',D')$-transcendental at $\cU$.
\item[(2)] If $B' \in O(B)$ and $D' \in O(D)$, then every everywhere $(B,D)$-transcendental element of $\cA$ is everywhere $(B',D')$-transcendental.
\ei
\end{prp}

\begin{proof}
The assertion (1) follows from $P_{B',D',\cU} \subset P_{B,D,\cU}$, and the assertion (2) follows from the assertion (1) by Proposition \ref{domination}.
\end{proof}

\begin{crl}
\label{generalised hierarchy}
Let $(B,D) \in (\R_{\geq 0}^{\bP})^2$.
\bi
\item[(1)] For any $\cU \in \beta \bP \setminus \bP$, if $\lim_{p \to \cU} B(p) > 0$ and $\lim_{p \to \cU} D(p) > 0$, then every element of $\cA$ $(B,D)$-transcendental at $\cU$ is transcendental at $2^{\aleph}$-many $\phi \in \Spec(\cA)(\C)$.
\item[(2)] If $B(p) > 0$ and $D(p) > 0$ for all but finitely many $p \in \bP$, then every everywhere $(B,D)$-transcendental element of $\cA$ is everywhere transcendental.
\ei
\end{crl}

\begin{proof}
The assertion (1) follows from Proposition \ref{characterisation of Spec} (1), the second assertion of Proposition \ref{quasi-finite}, Proposition \ref{transcendence implies naive transcendence}, Proposition \ref{BD-polynomial}, and Corollary \ref{hierarchy} (1). The assertion (2) follows from Proposition \ref{characterisation of Spec} (2), Proposition \ref{BD-polynomial}, and Corollary \ref{hierarchy} (2).
\end{proof}

Thus we obtained a hierarchy of transcendence not only generalising but also strengthening the naive transcendence.

\section{Transcendence Criteria}
\label{Transcendence Criteria}

In this section, we give criterion of the transcendence formulated in \S \ref{Local Structure}. We note that any criterion of naive transcendence such as \cite{MS26} Lemma 2.5 (cf.\ proof of \cite{LW25-2} Theorem 1) and \cite{MS26} Lemma 2.7 (cf.\ proof of \cite{LW25-1} Theorem 1.3) directly gives a criterion of the existence of $2^{\aleph}$-many points witnessing transcendence by Proposition \ref{transcendence implies naive transcendence}. In order to avoid such obvious extensions, we give criteria of transcendence extending \cite{MS26} Lemma 2.1 (cf.\ \cite{AF24} Proposition 3.7) and \cite{MS26} Lemma 2.3 in less trivial ways by weakening the assumption or strengthening the conclusion.

\begin{prp}[Extension of \cite{MS26} Lemma 2.1]
\label{criterion 1}
Let $f \in \prod_{p \in \bP} \Fp$ and $C \in \R_{\geq 0}^S$ with $S \in \cP_{= \aleph_0}(\N)$ and $x^d \in O(C(x))$ for any $d \in \N$. Set $S' \coloneqq \set{n \in S}{\exists p \in \bP_{\geq C(n)}[f(p) = n + p \Z]}$. If $\# S' = \aleph_0$, then the image of $f$ in $\cA$ is transcendental at $2^{\aleph}$-many $\phi \in \Spec(\cA)(\C)$.
\end{prp}

We note that $C(x) = 2^x$ witnesses the setting under the condition of \cite{MS26} Lemma 2.1. Unlike \cite{MS26} Lemma 2.1, we are only assuming the existence of a single $p \in \bP_{\geq C(n)}$ for each $n \in S$.

\begin{proof}
Take a choice map $\pi \colon S' \to \bP$ of $(\set{p \in \bP_{\geq C(n)}}{f(p) = n + p \Z})_{n \in S'}$. By $x \in O(C(x))$, we have $\# \set{n \in S}{p \geq C(n)} < \aleph_0$ for any $p \in \bP$, and hence $\pi$ is finite-to-one. Therefore, replacing $S$ by an infinite subset if necessary, we may assume the injectivity of $\pi$.

\vs
We denote by $\ol{f} \in \cA$ the image of $f$, and by $X \subset \Spec(\cA)$ the clopen subset corresponding to $\set{\cU \in \beta \bP \setminus \bP}{\pi(S') \in \cU} \subset \beta \bP \setminus \bP$ by the homeomorphism in Proposition \ref{characterisation of Spec} (2). Set $Y \coloneqq \set{\phi \in \Spec(\cA)(\cC)}{\ker(\phi) \in X}$. We have $\# Y = 2^{\aleph}$ by $\# \pi(S') = \aleph_0$ (cf.\ the proof of Proposition \ref{transcendence implies naive transcendence}). It suffices to show that $\ol{f}$ is transcendental at any $\phi \in Y$.

\vs
Let $F \in \Q[x] \setminus \ens{0}$. It suffices to show $F(f) \notin \ker(\phi)$. Take $r \in \Z_{> 0}$ with $rF \in \Z[x]$. We denote by $\bP_0 \in \cP_{< \aleph_0}(\bP)$ the set of prime factors of $r$. Set $S'' \coloneqq \set{n \in S'}{F(n) = 0}$. Since $F$ has at most finite zeros in $\Q$, we have $\# S'' < \aleph_0$. By $\v{rF(x)} \in O(x^{\deg(F)})$ and $x^{\deg(F)+1} \in O(C)$, there exists some $n_0 \in S'$ such that for any $n \in S'_{\geq n_0}$, $n \notin S''$ and $\max (\bP_0 \cup \ens{\v{rF(n)}}) < C(n)$ hold. This implies $\pi(n) \notin \bP_0$ and $\v{rF(n)} < \pi(n)$, i.e.\ $F \in \Z_{(\pi(n))}[x]$ and $F(f(\pi(n))) \not\equiv 0 \pmod{\pi(n)}$, for any $n \in S'_{\geq n_0}$. We obtain $e_{\pi(S'_{\geq n_0})} \in F(f) \cA$.

\vs
By $\phi \in Y$, there exists some $\cU \in \beta \bP \setminus \bP$ such that $\pi(S') \in \cU$ and $\ker(\phi)$ is generated by $\set{1 - e_U}{U \in \cU}$, where $e_U \in \cA$ denotes the image of the characteristic function of $U$ for each $U \in \cU$. Since $\pi(S') \setminus \pi(S'_{\geq n_0})$ is a finite set, we have $\pi(S'_{\geq n_0}) \in \cU$, and hence $1 - e_{\pi(S'_{\geq n_0})} \in \ker(\phi)$. This implies $F(f) \notin \ker(\phi)$ by $e_{\pi(S'_{\geq n_0})} \in F(f) \cA$.
\end{proof}

\begin{exm}
For each $n \in \N$, we set $p_n \coloneqq \min \bP_{\geq 2^{n+1}}$. By Bertrand's postulate, the assignment $n \mapsto p_n$ defines an injective map $\N \hookrightarrow \bP$. For any $f \in \Z^{\bP}$ with $f(p_n) = n$ for any $n \in \N$, the image of $(f(p) + p \Z)_{p \in \bP}$ in $\cA$ is transcendental at $2^{\aleph}$-many $\phi \in \Spec(\cA)(\C)$ by Proposition \ref{criterion 1} applied to $C(x) = 2^{x+1}$, and hence is naively transcendental by Proposition \ref{transcendence implies naive transcendence}.
\end{exm}

\begin{prp}[Extension of \cite{MS26} Lemma 2.3]
\label{criterion 2}
Let $f \in \prod_{p \in \bP} \Fp$, $C \in \Z^{\bP}$, and $(B,D) \in (\R_{\geq 0}^{\bP})^2$.
\bi
\item[(1)] For any $\cU \in \beta \bP \setminus \bP$, if $B + 1 \in_{\cU} o(\v{C})$ and $B(x)(D(x) + 1)\v{C(x)}^{dD(x)} \in_{\cU} o(x)$ for any $d \in \N$ and $\lim_{p \to \cU} f(p) = \lim_{p \to \cU} (C(p) + p \Z)$, then the image of $f$ in $\cA$ is $(B,D)$-transcendental at $\cU$.
\item[(2)] If $B + 1 \in o(\v{C})$ and $B(x)(D(x) + 1)\v{C(x)}^{dD(x)} \in o(x)$ for any $d \in \N$ and $f(p) = C(p) + p \Z$ for all but finitely many $p \in \bP$, then the image of $f$ in $\cA$ is everywhere $(B,D)$-transcendental.
\ei
\end{prp}

We note that $(C,B,D) = ((a_p)_{p \in \bP},1,1)$ witnesses the setting under the condition of \cite{MS26} Lemma 2.3 by Proposition \ref{BD-polynomial}. Unlike \cite{MS26} Lemma 2.3, our conclusion includes a result on the hierarchy of transcendence.

\begin{proof}
The assertion (2) follows from the assertion (1) and Proposition \ref{domination}. We show the assertion (1). Let $F \in P_{B,D,\cU}$ with $\lim_{p \to \cU} F(p) \neq 0$. By $\deg(F) \in_{\cU} O(D)$, there exists some $d \in \N$ such that $\deg(F) \leq_{\cU} dD$. By $\n{F} \in_{\cU} O(B)$ and $B(x) (D(x) + 1) \v{C(x)}^{d'D(x)} \in_{\cU} o(x)$ for any $d' \in \ens{0,d}$, we have $2 \n{F}(x)(\deg(F(x)) + 1) \max \ens{1,\v{C(x)}^{\deg(F(x))}} \leq_{\cU} x$. By $\n{F} \in_{\cU} O(B)$ and $B + 1 \in_{\cU} o(\v{C})$, we have $2 \n{F} + 2 \leq_{\cU} \v{C}$. Therefore, there exists some $U \in \cU$ such that $F(p) \neq 0$, $\n{F}(p)(\deg(F(p)) + 1)\max \ens{1,\v{C(p)}^{\deg(F(p))}} < p$, and $\n{F}(p) + 1 < \v{C(p)}$ for any $p \in U$.

\vs
Let $p \in U$. Set $F_p \coloneqq \sum_{d=0}^{\infty} R_p(a_d(F(p))) x^d \in \Z[x]$. We have
\be
\v{F_p(C(p))} \leq \sum_{d=0}^{\deg(F(p))} \v{R_p(a_d(F(p)))} \v{C(p)}^d \leq \n{F}(p)(\deg(F(p)) + 1)\max \ens{1,\v{C(p)}^{\deg(F(p))}} < p.
\ee
If $\deg(F(p)) = 0$, then we have $F_p(C(p)) = a_0(F_p) \neq 0$ by $F(p) \neq 0$. If $\deg(F(p)) > 0$, then we have
\be
\max{}_{d=0}^{\deg(F(p))-1} \v{\frac{R_p(a_d(F(p)))}{R_p(a_{\deg(F(p))}(F(p)))}} + 1 \leq \n{F}(p) + 1 < \v{C(p)}
\ee
and hence $F_p(C_p) \neq 0$ by Cauchy's bound of the absolute value of a real root. Therefore, we conclude $\v{F_p(C_p)} \in \Z \cap [1,p-1]$, and hence $F(p)(f(p)) \neq 0$.

\vs
We obtain $\set{p \in \bP}{F(p)(f(p)) \neq 0} \in \cU$, and hence $(\lim_{p \to \cU} F(p))(\lim_{p \to \cU} f(p)) \neq 0$. This implies that $f$ is $(B,D)$-transcendental at $\cU$.
\end{proof}

\begin{exm}
Let $\alpha \in \R_{> 0}$. The image of $(\lfloor (\log p)^{\alpha} \rfloor + p \Z)_{p \in \bP}$ in $\cA$ is everywhere $((\log x)^{\beta},\log \log x)$-transcendental for any $\beta \in (0,\alpha)$ by Proposition \ref{criterion 2}, and hence is naively transcendental by Proposition \ref{transcendence implies naive transcendence} and Proposition \ref{generalised hierarchy}.
\end{exm}

\section{Relation to Prime Gaps}
\label{Relation to Prime Gaps}

In this section, we observe a relation between asymptotic behaviour of prime gaps and transcendence in $\cA$. We define a map $\sigma \colon \bP \to \bP$ by
\be
\sigma(p) \coloneqq \min \bP_{> p}.
\ee
For $k \in \N$, we denote by $\sigma^k$ the $k$-th iterated composition of $\sigma$. In the view that $\sigma$ plays a role analogous to the successor function on $\N$, $\sigma^k$ for $k \in \N$ plays a role analogous to the parallel shift $n \mapsto n + k$ on $\N$. By $\sigma^0 = \id_{\bP}$, the image of $(\sigma^0(p) + p \Z)_{p \in \bP}$ in $\cA$ is $0$. How about a higher iteration? We give a conditional answer to the question.

\vs
In order to state the answer, we prepare convention. We define a map $\Delta \colon \bP \to \N_{> 0}$ by
\be
\Delta(p) \coloneqq \sigma(p) - p,
\ee
and a map $\epsilon \colon \bP \to \R_{> 1}$ by
\be
\epsilon(p) \coloneqq \log_{\log p} \Delta(p).
\ee
We consider the following conjecture:

\begin{conj}
\label{weak Cramer conjecture}
The following relation holds:
\be
\lim_{p \in \bP} \frac{\epsilon(p) \log \log p}{\log p} = 0
\ee
\end{conj}

We note that Cram\'er's conjecture (cf.\ \cite{Cra36} (4)) stating $\Delta(p) \in O((\log p)^2)$ immediately implies
\be
\limsup_{p \in \bP} \epsilon(p) \leq 2,
\ee
and hence Conjecture \ref{weak Cramer conjecture} follows from Cram\'er's conjecture. Under Conjecture \ref{weak Cramer conjecture}, we obtain the naive transcendence of the image of the iterated shift $\sigma^k$ in $\cA$ for $k \in \N_{> 0}$.

\begin{thm}
For any $k \in \N_{> 0}$, the image of $(\sigma^k(p) + p \Z)_{p \in \bP}$ in $\cA$ is transcendental at $2^{\aleph}$-many $\phi \in \Spec(\cA)(\C)$ under Conjecture \ref{weak Cramer conjecture}.
\end{thm}

\begin{proof}
For each $d \in \N$, we denote by $\pi_d \in \bP$ the least prime number $p$ satisfying
\be
\sup_{p' \in \bP_{\geq p}} \frac{\epsilon(p) \log \log p}{\log p} < \frac{1}{d+1},
\ee
which exists by Conjecture \ref{weak Cramer conjecture}. By the existence of prime deserts, $\set{\sigma^k(p) - p}{p \in \bP}$ is unbounded. Therefore, the following recursive relation uniquely defines a strictly increasing sequence $s \colon \N \to \N$:
\be
s(n) \coloneqq \min \left( \set{\sigma^k(p) - p}{p \in \bP_{\geq \pi_n}} \setminus \set{s(i)}{i \in \N_{< n}} \right)
\ee
Set
\be
S \coloneqq \im(s).
\ee
By the injectivity of $s$, we have $\# S = \aleph_0$. We denote by $s' \colon S \to \N$ the inverse of the restriction $\N \to S$ of $s$, and define a map $C \colon S \to \R_{\geq 0}$ by
\be
C(n) \coloneqq \min \set{p \in \bP_{\geq \pi_{s'(n)}}}{\sigma^k(p) = n},
\ee
which makes sense by the definition of $S$. We show $x^d \in O(C(x))$ for any $d \in \N$. Set $n_0 \coloneqq \min \set{n \in S_{\geq 2^k k^{d+1}}}{s'(n) \geq d}$. Let $n \in S_{\geq n_0}$. We have $C(n) \in \bP_{\geq \pi_{s'(n)}}$ by the definition of $C$, and hence $\sigma^j(C(n)) \geq \pi_{s'(n)} \geq \pi_d$ for any $j \in \N$. For any $p \in \bP_{\geq \pi_d}$, we have
\be
\frac{\epsilon(p) \log \log p}{\log p} < \frac{1}{d+1}
\ee
by the definition of $\pi_d$, and hence
\be
\Delta(p) = (\log p)^{\epsilon(p)} < (\log p)^{\frac{\log p}{(d+1) \log \log p}} = p^{\frac{1}{d+1}},
\ee
i.e.\
\be
\Delta(p)^{d+1} < p.
\ee
This implies
\be
n^{d+1} & = & (\sigma^k(C(n)) - C(n))^{d+1} = \left( \sum_{j=0}^{k-1} \sigma^{j+1}(C(n)) - \sigma^j(C(n)) \right)^d = \left( \sum_{j=0}^{k-1} \Delta(\sigma^j(C(n))) \right)^{d+1} \\
& \leq & \left( k \max{}_{j=0}^{k-1} \Delta(\sigma^j(C(n))) \right)^{d+1} = k^{d+1} \max{}_{j=0}^{k-1} \Delta(\sigma^j(C(n)))^{d+1} < k^{d+1} \max{}_{j=0}^{k-1} \sigma^j(C(n)) \\
& < & 2^k k^{d+1} C(n),
\ee
and hence
\be
n^d < C(n)
\ee
by $n \geq n_0 \geq 2^k k^{d+1}$. This implies $x^d \in O(C(x))$.

\vs
For any $n \in S$, there exists some $p \in \bP_{\geq \pi_{s'(n)}}$ such that $\sigma^k(p) - p = n$ by the definition of $S$, and hence $p \geq C(n)$ by the definition of $C$. Therefore, the assertion follows from Proposition \ref{criterion 1} applied to $C$.
\end{proof}

We finish this paper with summary of preceding studies related to Conjecture \ref{weak Cramer conjecture}. To begin with, pigeonhole principle and prime number theorem imply
\be
\liminf_{p \in \bP} \frac{\Delta(p)}{\log p} \leq 1,
\ee
and hence
\be
\liminf_{p \in \bP} \epsilon(p) \leq 1.
\ee
The first inequality has ever been improved to be many inequalities whose right hand side is a positive constant smaller than $1$, and is finally extended by D.\ A.\ Goldston, J.\ Pintz, and C.\ Y.\ Yildirim to the equality
\be
\liminf_{p \in \bP} \frac{\Delta(p)}{\log p} = 0
\ee
in \cite{GPY09} Theorem 2. Also the gap of consecutive primes is studied well. The first trial on the ratio with $\log p$ is the conditional result
\be
\liminf_{p \in \bP} \frac{\sigma^2(p) - p}{\log p} = 0
\ee
under Elliott--Halberstam conjecture by them (cf.\ the paragraph after \cite{GPY09} Theorem 3). However, starting from Maynard--Tao theorem
\be
\liminf_{p \in \bP} \sigma^k(p) - p \ll k^3 e^{4k}
\ee
for any $k \in \N$ (cf.\ \cite{May15} Theorem 1.1), there have been many improvements of the right hand side of the inequality. We note that Maynard--Tao theorem implies that the left hand side of Twin Prime Conjecture
\be
\liminf_{p \in \bP} \Delta(p) = 2
\ee
is finite. Although Twin Prime Conjecture is still open, the estimation of upper bounds of the left hand side is being improved day-by-day.

\vs
We explained results on $\liminf$. Concerning $\limsup$, Cram\'er's theorem $\Delta(x) \in O(\sqrt{x} \log x)$ under Riemann hypothesis (cf.\ \cite{Cra36} Theorem 1) immediately implies
\be
\limsup_{p \in \bP} \frac{\epsilon(p) \log \log p}{\log p} \leq \frac{1}{2}
\ee
under the same hypothesis. R.\ C.\ Baker, G.\ Harman, and J.\ Pintz unconditionally showed $\Delta(x) \leq x^{0.535}$ in \cite{BHP01} Theorem 1, and the result immediately implies
\be
\limsup_{p \in \bP} \frac{\epsilon(p) \log \log p}{\log p} \leq 0.535.
\ee
There is also a result on $\limsup$ of $\epsilon(p)$ alone. Westzynthius's estimation
\be
\limsup_{p \in \bP} \frac{\Delta(p)}{\log p} = \infty
\ee
in \cite{Wes31} implies
\be
\limsup_{p \in \bP} \epsilon(p) \geq 1.
\ee
Westzynthius showed a slightly stronger result stating a lower bound of $\frac{\Delta(p)}{\log p}$ given by an explicit expression using iterated logarithms up to a constant, and there have ever been many improvements of the estimation (cf.\ \cite{Erd35} Theorem 1, \cite{Ran38} Theorem 1, \cite{May16} Theorem 1, \cite{FGKT16} Theorem 1). However, there is no unconditional result implying either one of
\be
\limsup_{p \in \bP} \epsilon(p) > 1
\ee
or
\be
\limsup_{p \in \bP} \epsilon(p) \leq 2
\ee
as far as we know. We note that the strict inequality
\be
\limsup_{p \in \bP} \epsilon(p) < 2,
\ee
is stronger than Cram\'er's conjecture, and conflicts the strong version
\be
\limsup_{p \in \bP} \frac{\Delta(p)}{(\log p)^2} = 1
\ee
of Cram\'er's conjecture (cf.\ the last paragraph of \cite{Cra36} \S I with $P_n = p_n$) implying
\be
\limsup_{p \in \bP} \epsilon(p) = 2.
\ee
At least, the finiteness of the left hand side implies Conjecture \ref{weak Cramer conjecture}, and is open as far as we know.

\vspace{0.3in}
\addcontentsline{toc}{section}{Acknowledgements}
\noindent {\Large \bf Acknowledgements}
\vspace{0.2in}

\noindent
I thank S.\ Seki for informing me of \cite{MS26}. It gave me an opportunity to write this survey article. I thank K.\ Tokimoto for reminding me of field theory. I thank all people who helped me to learn mathematics and programming. I also thank my family.


\addcontentsline{toc}{section}{References}
\begin{thebibliography}{99}

\bibitem[AF24]{AF24} T.\ Anzawa and H.\ Funakura, {\it Congruences for the $q$-Fibonacci Sequence Related to Its Transcendence}, The Ramanujan Journal, Volume 63, Issue 4, pp.\ 1057--1072, 2024.

\bibitem[AM69]{AM69} M.\ F.\ Atiyah and I.\ G.\ MacDonald, {\it Introduction To Commutative Algebra}, Addison-Wesley Series in Mathematics, Addison-Wesley Publishing Company, 1969.

\bibitem[BHP01]{BHP01} R.\ C.\ Baker, G.\ Harman, and J.\ Pintz, {\it The Difference between Consecutive Primes, II}, Proceedings of the London Mathematical Society, Volume 83, Issue 3, pp.\ 532--562, 2001.

\bibitem[BS69]{BS69} J.\ L.\ Bell and A.\ B.\ Slomson, {\it Models and Ultraproducts: An introduction}, North-Holland Publishing Company, 1969.

\bibitem[Cra36]{Cra36} H.\ Cram'er, {\it On the Order of Magnitude of the Difference between Consecutive Prime Numbers}, Acta Arithmetica, Volume 2, Number 1, pp.\ 23--46, 1936.

\bibitem[Eda82]{Eda82} K.\ Eda, {\it A Boolean Power and a Direct Product of Abelian Groups}, Tsukuba Journal of Mathematics, Volume 6, Number 2, pp.\ 187--194, 1982.

\bibitem[Eda83]{Eda83} K.\ Eda, {\it Almost-Slender Groups and Fuchs-44-Groups}, Comentarii Mathematici, Universitatis Sancti Pauli, Volume 32, Number 2, pp.\ 131--135, 1983.

\bibitem[E{\L}54]{EL54} A.\ Ehrenfeucht and J.\ {\L}o\'s, {\it Sur le Produits Cart\'esiens des Groupes Cycliques Infinis}, Bulletin de l'Academie Polonaise des Sciences, Volume III, Num\'ero 2, pp.\ 261--263, 1954.

\bibitem[Erd35]{Erd35} P.\ Erd\"os, {\it On the Difference of Consecutive Primes}, The Quarterly Journal of Mathematics, Volume 6, Issue 1, pp.\ 124--128, 1935.

\bibitem[FGKT16]{FGKT16} K.\ Ford, B.\ Green, S.\ Konyagin, and T.\ Tao, {\it Large Gaps between Consecutive Prime Numbers}, Annals of Mathematics, Volume 183, Issue 3, pp.\ 935--974, 2016.

\bibitem[Fuc70]{Fuc70} L.\ Fuchs, {\it Infinite Abelian Groups, Volume I}, Academic Press, 1970.

\bibitem[Fuc15]{Fuc15} L.\ Fuchs, {\it Abelian Groups}, Springer Monographs in Mathematics, Springer, 2015.

\bibitem[GPY09]{GPY09} D.\ A.\ Goldston, J.\ Pintz, and C.\ Y.\ Yildirim, {\it Primes in Tuples I}, Annals of Mathematics, Volume 170, Issue 2, pp.\ 819--862, 2009.

\bibitem[GRW81]{GRW81} R.\ G\"obel, S.\ V.\ Richkov, and B.\ Wald, {\it A General Theory of Slender Groups and fuchs-44-groups}, Abelian Group Theory, Lecture Notes in Mathematics, Volume 874, pp.\ 194-201, Springer, 1981

\bibitem[Iva78]{Iva78} A.\ V.\ Ivanov, {\it A Problem on Abelian Groups}, Mathematics of the USSR-Sbornik, Volume 34, Number 4, pp.\ 461--474, 1978.

\bibitem[KZ]{KZ} M.\ Kaneko and D.\ Zagier, {\it Finite Multiple Zeta Values}, to appear in the Proceedings of the $17$th MSJ-SI conference on Modular forms and Multiple Zeta values.

\bibitem[LW25-1]{LW25-1} F.\ Luca and W.\ Zudilin, {\it Irrationality and Transcendence Questions in the `Poor Man's Ad\`ele Ring'}, The Ramanujan Journal, Volume 67, Issue 4, Article number 88, 2025.

\bibitem[LW25-2]{LW25-2} F.\ Luca and W.\ Zudilin, {\it Poor Man's Transcendence for Frobenius Traces of Elliptic Curves}, arXiv:2507.14773, 2025.

\bibitem[Mar60]{Mar60} J.\ M.\ Maranda, {\it On Pure Subgroups of Abelian Groups}, Archly der Mathematik, Volume XI, pp.\ 1--13, 1960.

\bibitem[May15]{May15} J.\ Maynard, {\it Small Gaps between Primes}, Annals of Mathematics, Volume 181, Issue 1, pp.\ 383--413, 2015.

\bibitem[May16]{May16} J.\ Maynard, {\it Large Gaps between Primes}, Annals of Mathematics, Volume 183, Issue 3, pp.\ 915--933. 2016.

\bibitem[MS26]{MS26} T.\ Matsusaka and S.\ Seki, {\it Some Results on Naive Transcendence in the Ring of Integers Modulo Infinitely Large Primes}, arXiv:2604.25566, 2026.

\bibitem[Ngu24]{Ngu24} D.\ Q.\ N.\ Nguyen, {\it Higher Reciprocity Law and an Analogue of the Grunwald--Wang Theorem for the Ring of Polynomials over an Ultra-finite Field}, Annals of Pure and Applied Logic, Volume 175, Issue 6, article 103438, 2024.

\bibitem[Ran38]{Ran38} R.\ A.\ Rankin, {\it The Difference between Consecutive Prime Numbers}, Proceedings of the Edinburgh Mathematical Society, Volume 13, Issue 4, pp.\ 242--247, 1938.

\bibitem[Spe50]{Spe50} E.\ Specker, {\it Additive Gruppen von Folgen Ganzer Zahlen}, Portugaliae Mathematica, Volume 9, pp.\ 131--140, 1950.

\bibitem[Wes31]{Wes31} E.\ Westzynthius, {\it \"Uber die Verteilung der Zahlen, die zu den $n$ ersten Primzahlen teilerfremd sind}, Commentationes Physico-Mathematicae, Volume 5, Number 25, pp.\ 1--37, 1931.

\bibitem[Zee55]{Zee55} E.\ C.\ Zeeman, {\it On Direct Sums of Free Cycles}, Journal of the London Mathematical Society, Volume s1-30, Issue 2, pp.\ 195--212, 1955.

\end{thebibliography}
\end{document}